\documentclass[11pt,reqno]{amsart}
\usepackage{latexsym}
\usepackage{amssymb,amsfonts,amsmath,mathrsfs}
\usepackage{hyperref}
\newtheorem{theorem}{Theorem}[section]

\newtheorem{proposition}{Proposition}[section]

\newtheorem{lemma}{Lemma}[section]

\newtheorem*{theorem1.2}{Theorem 1.2}
\newtheorem*{theorem1.3}{Theorem 1.3}
\newcommand{\tmop}[1]{\ensuremath{\operatorname{#1}}}
\numberwithin{equation}{section}
\def\pamod{\! \! \! \! \pmod}

\begin{document}

\title[On the distribution of the divisor function  and Hecke eigenvalues]
{On the distribution of the divisor function \\ and Hecke eigenvalues}

\author[S. Lester and N. Yesha]{Stephen Lester and Nadav Yesha}

\address{School of Mathematical Sciences, Tel Aviv University,
Tel Aviv, 69978, Israel }

\email{slester@post.tau.ac.il}

\address{School of Mathematical Sciences, Tel Aviv University,
Tel Aviv, 69978, Israel }

\email{nadavye1@post.tau.ac.il}

\date{}

\begin{abstract}
We investigate the behavior
of the divisor function in both short intervals and in arithmetic progressions.
The latter problem was recently studied by 
\'E. Fouvry, S. Ganguly, E. Kowalski, and Ph. Michel. We prove a complementary
result to their main theorem.
We also show that in short
intervals of certain lengths the divisor function has a Gaussian limiting distribution. 
The analogous problems
for Hecke eigenvalues are also
considered.
\end{abstract}



\maketitle

\section{Introduction}
The study of the behavior of
\[
\Delta(x)=\sum_{n \le x} d(n)-x\big(\log x+2\gamma-1 \big),
\]
where $\gamma$ is Euler's constant is a classical topic in analytic number theory. For instance,
Dirichlet's divisor problem asks for the smallest $\alpha$ such that
\[
\Delta(x) \ll x^{\alpha+\varepsilon},
\]
for all $\varepsilon>0$.
Dirichlet showed that $\alpha \le 1/2$, which was sharpened by Voronoi \cite{Voronoi} who proved  
$\alpha \le 1/3$. A more  recent result of Huxley \cite{Huxley} gives $\alpha \le 131/416$.
On the other hand, Hardy \cite{Hardy} proved that $\Delta(x)=\Omega((x \log x)^{1/4} \log \log x)$,
and it is conjectured that $\alpha=1/4$.

In this article we study the average behavior of $d(n)$ on two different sparse sets, namely,
short intervals and arithmetic progressions modulo a large prime number. The similarities
between these two problems are striking and in the analogous problems for function fields 
over a finite field there
is a fundamental identity that clarifies this connection in that setting (see Lemma 4.2 of \cite{RudnicKeating} for a similar identity).

\subsection{The divisor function in arithmetic progressions}
The behavior of divisor function
in an arithmetic progression
has been studied by numerous authors.
For instance,
Blomer \cite{Bl} and 
Lau and Zhao \cite{ZL} have 
investigated the variance of sums of the divisor function in progressions. Notably, Lau and Zhao prove
an asymptotic formula
for the variance of sums of the divisor function $d(n)$ with $1 \le n \le X$
in arithmetic progressions modulo $q$,
for $q$ satisfying $X^{1/2}<q<X^{1-\epsilon}$.

Instead of working with a sum of $d(n)$
over $1 \le n \le X$ it is technically advantageous
to consider smoothed sums of the form
\[
 \sum_{n} d(n) w\Big(\frac{n}{X} \Big),
\]
where $w$ is a smooth function compactly supported on the positive real numbers.
Recently, \'E. Fouvry, S. Ganguly, E. Kowalski, and Ph. Michel ~\cite{FGKM}
studied the distribution of smoothed sums of $d(n)$ over arithmetic progressions
modulo a prime number, $p$ and showed that it has a Gaussian limiting distribution as $p \rightarrow \infty$, for $p^{2-\epsilon} \ll X =o(p^{2})$.

To state their result more precisely, let $w$ be a real-valued smooth function compactly supported in the positive real numbers. For a prime $p$, define
\[
\mathcal S_{d}(X,p,a;w)=\sum_{\substack{n \ge 1 \\ n \equiv a  \pamod p} } d(n) w\Big(\frac{n}{X} \Big),
\]
and
\[
\mathcal M_{d}(X,p;w)=\frac1p \sum_{n \ge 1} d(n)w\Big(\frac{n}{X}\Big)- \frac{1}{p^2}\int_0^{\infty} (\log y+2\gamma-2\log p)w\Big(\frac{y}{X}\Big) \, dy.
\]
Also, let
\[
E_{d}(X,p,a;w) = \frac{\mathcal S_{d}(X,p,a;w)-\mathcal M_{d}(X,p;w)}{ \lVert w \rVert_{L^2} \sqrt{2 \pi^{-2} (X/p) \cdot \log^3( p^2/X+2)}},
\]
Theorem 1.1 of \cite{FGKM}
states that for $X=p^2/\Phi(p)$, where $\Phi(p)\ge 1$
is a function that tends to infinity with $p$
in a way such that $\log \Phi(p) =o(\log p)$, 
 as
 $p \rightarrow \infty$, 
\begin{equation} \label{Theorem fgkm}
 \frac{1}{p-1} \# \bigg\{ a \in \mathbb F_p^{\times} : \alpha \le E_{d}(X,p,a;w)  
\le \beta\bigg\}=\frac{1}{\sqrt{2\pi}} \int_{\alpha}^{\beta} \! e^{-x^2/2} \, dx+o(1).
\end{equation}

This theorem is proved by calculating the moments of $E_{d}(X,p,a;w)$. These
moments are estimated through an application of the Voronoi summation formula
along with estimates of moments
of Kloosterman sums due to N. Katz. 

We are interested in seeing if
the smooth weight function can be replaced
with a sharp cut-off function in this result.
This is because in some cases smoothing completely
alters the nature of the problem.
For instance, smoothing 
substantially changes the Dirichlet
divisor problem. One can prove for any $A \ge 1$ and a smooth function $w$
that is compactly supported on the positive real numbers that
\[
\sum_{n \ge 1} d(n) w\Big( \frac{n}{X} \Big)
=\int_0^{\infty} (\log y+2\gamma) w\Big( \frac{y}{X} \Big) \, dy+O(X^{-A}),
\]
where the implied constant depends on $w$ and $A$. Thus,
the remainder term is very small, unlike that of
in the original Dirichlet divisor problem so that 
the smooth weight function cannot
be replaced with a sharp cut-off function here.  
Moreover, something similar happens in the case of
smoothed sums of the divisor function in arithmetic progressions modulo a prime number $p$ when $p$ is small
relative to $X$. Here one can show that $E_d(X,p,a;w) \ll X^{-A}$ if $X^{\varepsilon}<p<X^{1/2-\varepsilon}$
(see Lemma \ref{voronoi} below). From 
this we see
that understanding the distribution of $E_d(X,p, a;w)$ in this regime is trivial since it is always smaller than
any negative power of $X$. This is an effect of smoothing (see Theorem 4 of ~\cite{ZL}),
and it would be interesting to study the distribution 
of a sharp cut-off analog of $E_d(X,p, a;w)$ in this regime.

We show that an analog of Theorem 1.1 of \cite{FGKM} holds for sums of the divisor function with sharp cut-offs. 
The process of removing the weight function is subtle.
Our method requires the existence of a compactly supported function $w_p$, that depends on $p$, such
that: $1)$ $w_p$ approximates $\mathbf 1_{[0,1]}$
as $p \rightarrow \infty$ in a suitably strong sense; $2)$
a certain integral transform of $w_p$, arising
from the Voronoi summation formula, decays sufficiently rapidly.


\begin{theorem} \label{sharp divisor}
For a prime $p$, let
\[
\mathcal S_{d}(X,p,a)=\sum_{\substack{1 \le n \le X \\ n \equiv a  \pamod p} } d(n) ,
\]
and
\[
\mathcal M_{d}(X,p)=\frac1p \sum_{1 \le n \le X} d(n)- \frac{X}{p^2}\Big(\log X-1+2\gamma-2\log p\Big).
\]
Also, let
\[
E_{d}(X,p,a) = \frac{\mathcal S_{d}(X,p,a)-\mathcal M_{d}(X,p)}{\sqrt{2 \pi^{-2} (X/p) \cdot \log^3( p^2/X+2)}}.
\]
For $X=p^2/\Phi(p)$, where $\Phi(p)\ge 1$
is a function that tends to infinity with $p$
in a way such that $\log \Phi(p) =o(\log p)$. As
 $p \rightarrow \infty$, we have
\[
 \frac{1}{p-1} \# \bigg\{ a \in \mathbb F_p^{\times} : \alpha \le   E_{d}(X,p,a) 
\le \beta\bigg\}=\frac{1}{\sqrt{2\pi}} \int_{\alpha}^{\beta} \! e^{-x^2/2} \, dx+o(1).
\]
\end{theorem}

Let $d_k(n)$ be the number of ways
of writing $n$ as a product of $k$ factors.
E. 
Kowalski and G. Ricotta \cite{KR} prove an analog of Theorem
1.1 of \cite{FGKM}
for $d_k(n)$ for any integer
$k \ge 3$. We have not succeeded in removing the smooth weight
 for any $k  > 2$.

\'E. Fouvry, S. Ganguly, E. Kowalski, and Ph. Michel also prove an analog
of \eqref{Theorem fgkm} for holomorphic Hecke cusp forms of weight $k$ and level one (see Corollary 1.4 of \cite{FGKM}).
We prove an analog of that result for a sharp cut-off
function as well.

Let $f$ be a primitive (Hecke eigenform) cusp form of even weight
$k$ and level 1, and consider its Fourier expansion
\[
f\left(\tau\right)=\sum_{n=1}^{\infty}\rho_f\left(n\right)n^{\frac{k-1}{2}}e\left(n\tau\right)
\]
where $f$ is normalized so that $\rho_f\left(1\right)=1$, so $\rho_f\left(n\right)$ is the eigenvalue of the (suitably normalized) Hecke operator $T\left(n\right)$.
Let $c_{f}=\frac{\left(4\pi\right)^{k}}{\Gamma\left(k\right)}\left\Vert f\right\Vert ^{2}.$ Here we used the notation
\[
\left\Vert f\right\Vert ^{2}=\frac{3}{\pi}\iint y^{k}\left|f\left(x+iy\right)\right|^{2} \frac{\mbox{d}x\mbox{d}y}{y^2}
\]
where the integral is taken over any fundamental domain for $SL_{2}\left(\mathbb{Z}\right)$.

\begin{theorem} \label{sharp cusp forms}
Let $f$ be a Hecke cusp form of weight $k$ and level one.
For a prime $p$, let
\[
\mathcal S_{f}(X,p,a)=\sum_{\substack{1 \le n \le X \\ n \equiv a  \pamod p} } \rho_f(n) ,
\]
and
\[
\mathcal M_{f}(X,p)=\frac1p \sum_{1 \le n \le X} \rho_f(n).
\]
Also, let
\[
 E_{f}(X,p,a) = \frac{\mathcal S_{f}(X,p,a)-\mathcal M_{f}(X,p)}{(c_f \cdot X/p)^{1/2}}.
\]
For $X=p^2/\Phi(p)$, where $\Phi(p)\ge 1$
is a function that tends to infinity with $p$
in a way such that $\log \Phi(p) =o(\log p)$. As
 $p \rightarrow \infty$, we have
\[
 \frac{1}{p-1} \# \bigg\{ a \in \mathbb F_p^{\times} : \alpha \le   E_{f}(X,p,a) 
\le \beta\bigg\}=\frac{1}{\sqrt{2\pi}} \int_{\alpha}^{\beta} \! e^{-x^2/2} \, dx+o(1).
\]
\end{theorem}

\subsection{The divisor function in short intervals} 
Heath-Brown studied the distribution of the normalized remainder term $x^{-1/4}\Delta\left(x\right)$ as $x\to\infty$, and proved that it has a limiting distribution function \cite{Heath-Brown}. The behavior of the remainder term for the divisor problem in short intervals was studied by several authors (cf. \cite{Jutila2,Ivic}). For example, Ivi\'{c} proved in \cite{Ivic} an asymptotic formula for the second moment of $\Delta\left(x+U\right)-\Delta\left(x\right)$, where $T^\varepsilon \le U=U\left(T\right)\leq T^{1/2-\varepsilon}$:
\[
\frac{1}{T}\int_{T}^{2T}\left(\Delta\left(x+U\right)-\Delta\left(x\right)\right)^{2}\mbox{d}x\sim\frac{8}{\pi^{2}}U\log^{3}\left(\frac{\sqrt{T}}{U}\right).
\]

We study the distribution of the sums of the divisor function in short intervals of the form $\left[x,x+\sqrt{x}/L\right]$, where $L$ grows to infinity in 
a way such that $\log L=o(\log x)$.
\begin{theorem} \label{Main theorem} Let $L=L\left(T\right)\to\infty$ as $T\to\infty$,
with the condition that $\log L=o(\log T)$.
 For $\alpha<\beta$ as $T \rightarrow \infty$ we have
\[
\frac1T \tmop{meas}\Big\{ x \in [T, 2T] : \alpha \le  \frac{ \Delta\Big(x+\sqrt{x}/L\Big)-\Delta(x)}{x^{1/4} \sqrt{\frac{8}{\pi^2} \frac{\log^3 L}{L}}} \le \beta \Big\}
=\frac{1}{\sqrt{2\pi}}\int_{\alpha}^{\beta} \, e^{-t^2/2} \, dt+o(1).
\]
\end{theorem}
We remark that the analogous problem for the circle problem, i.e. the distribution of lattice points in thin annuli was studied by Hughes and Rudnick \cite{HR} -- the corresponding normalized remainder in this case has also a Gaussian limiting distribution

Next, we prove the analogous result for the distribution of the sum of Hecke eigenvalues
in short intervals.
Let
\[
A_{f}\left(x\right)=\sum_{n\leq x}\rho_f\left(n\right).
\]

\begin{theorem} \label{Main theorem Hecke}
 For fixed $\alpha<\beta$ as $T \rightarrow \infty$
\[
\frac1T \tmop{meas}\Big\{ x \in [T, 2T] : \alpha \le  
\frac{A_f(x+\sqrt{x}/L)-A_f(x)}{x^{1/4} (c_f \, / L)^{1/2}} \le \beta \Big\}
=\frac{1}{\sqrt{2\pi}}\int_{\alpha}^{\beta} \, e^{-t^2/2} \, dt+o(1).
\]

\end{theorem}

We can also prove an analog of this result for the Hecke eigenvalues of Maass forms for the full modular group.
The proof requires some additional steps since the Ramanujan bound is not known in this case.


%
%
%
\section{The distribution of the divisor function in progressions}

\subsection{Preliminary Lemmas}

For a smooth function $g$ let
\begin{equation}
\mathcal B_{d}(g)(\xi)=
\begin{cases} \displaystyle
-2\pi \int_0^{\infty} \, g(u) Y_0(4\pi \sqrt{\xi u} )\, du \quad \mbox{ if } \quad \xi>0,\\
\displaystyle 4\int_0^{\infty} \, g(u) K_0(4\pi \sqrt{|\xi| u}) \, du \quad \mbox{ if } \quad  \xi<0.
\end{cases}
\end{equation}
Also, let
\[
\mathcal B_{f}(g)(\xi)=2\pi i^k \int_0^{\infty} \, g(u) J_{k-1}(4 \pi \sqrt{\xi u}) \, du  \quad \mbox{ for } \quad \xi>0.
\]

For $\delta>0$
let $w_{\delta}(x)$ be a non negative
smooth function supported on $[\delta,1]$
and equal to one on $[2\delta,1-\delta]$
that also satisfies $w_{\delta}^{(\ell)}(x) \ll 1/\delta^{\ell}$.

\begin{lemma} \label{Bessel}
Let $A \geq 1$ be an integer. For $\xi \ne 0$ we have
\[
\mathcal B_d(w_{\delta})(\xi) \ll \min\Big\{1+|\log |\xi||, |\xi|^{-A/2-1/4} \delta^{1-A} \Big\},
\]
where the implied constant depends on $A$.
Additionally, for $\xi>0$ 
\[
\mathcal B_f(w_{\delta})(\xi) \ll \min\Big\{1, \xi^{-A/2-1/4} \delta^{1-A} \Big\}
\]
where the implied constant depends on $A$.
\end{lemma}
\begin{proof}
Note that
\[
J_{k-1}(x) \ll 1, \qquad Y_0(x) \ll 1+|\log x|,  \qquad \mbox{ and }  \qquad K_0(x) \ll 1+|\log x|,
\]
which establishes the first bound for both claims.

We will only prove the second bound for $\mathcal B_f(w_{\delta})$; the 
proof of the bound for $\mathcal B_d(w_{\delta})$ follows from
a similar argument. By the change of variables $v= 4 \pi \sqrt{\xi u}$ we have
\[
\mathcal B_f(w_{\delta})(\xi)= \frac{i^k}{4 \pi \xi} \int_0^{\infty} w_{\delta}\Big(\frac{v^2}{16 \pi^2 \xi}\Big) v J_{k-1}(v) \, dv.
\]
Next set $\alpha=(16 \pi^2 \xi)^{-1}$ and note that
\[
\frac{d}{dx} \Big(x^{\nu+1}J_{\nu+1}(x)\Big)=x^{\nu+1}J_{\nu}(x),
\]
(see \cite{GR} equation 8.472.3).
Thus, integration by parts gives
\[
\mathcal B_f(w_{\delta})(\xi)= \frac{i^k}{4 \pi \xi} \int_0^{\infty}  v^{1-k} w_{\delta}(\alpha v^2)  \, d( v^k J_k (v))
=\frac{-i^k}{4 \pi \xi} \int_0^{\infty} (2\alpha v^2  w_{\delta}'(\alpha v^2) + (1-k)w_{\delta}(\alpha v^2))J_k(v) \, dv.
\]
Repeatedly integrating by parts, we see that 
\[
\mathcal B_f(w_{\delta})(\xi)
\ll \frac{1}{\xi} \int_0^{\infty} v^{-A+1} |J_{k-1+A} (v)| \sum_{\ell=0}^A (\alpha v^2)^{\ell} |w_{\delta}^{(\ell)}(\alpha v^2)|   \, dv.
\]
We now use the bound $J_{k-1+A}(x) \ll x^{-1/2}$ then make the
change of variables $y=\alpha v^2$ to get
\[
\mathcal B_f(w_{\delta})(\xi)
\ll \xi^{-A/2-1/4} \int_0^{\infty} y^{-A/2-1/4} \sum_{\ell=0}^A  y^{\ell} |w_{\delta}^{(\ell)}(y)| \, dy.
\]

Note that
$
w_{\delta}^{(\ell)}( x) \ll 1/\delta^{\ell}
$
and $w^{(\ell)}(x)$ is supported on the interval $[\delta, 2\delta] \cup [1-\delta, 1]$, for $\ell \ge 1$.
Hence, for $\ell \ge 1$ we have
\[
\int_0^{\infty} y^{\ell-A/2-1/4} |w_{\delta}^{(\ell)}(y)| \, dy
\ll \int_{\delta}^{2 \delta} y^{-A/2-1/4}  \, dy
+\delta^{-\ell} \int_{1-\delta}^{1} 1 \, dy \ll \delta^{3/4-A/2}+\delta^{1-\ell} \ll \delta^{1-A}. 
\]
For $\ell=0$ we have 
\[
\int_0^{\infty} y^{-A/2-1/4} |w_{\delta}(y)| \, dy \ll \int_{\delta}^{1} y^{-A/2-1/4} \, dy
\ll \delta^{3/4-A/2}+1 \ll \delta^{1-A}.
\]
The result follows by collecting estimates.
\end{proof}
\begin{lemma}\label{deriv est}
We have for $\xi \ne 0$.
\[
\Big(\mathcal B_{d}(w_{\delta})(\xi) \Big)' \ll |\xi|^{-5/4}.
\]
Additionally, for $\xi>0$ and $k\ge 2$ we have
\[
\Big(\mathcal B_f(w_{\delta})(\xi) \Big)' \ll |\xi|^{-5/4}.
\]
\end{lemma}
\begin{proof}
For the first claim differentiate inside the integral and use the formula
\[
\frac{d}{dx} Z_0(x)=-Z_1(x),
\]
for $Z=Y$ or $Z=K$ (see \cite{GR} equations 8.473.6 and 8.486.18).
Now argue as in the previous proof, but integrate by parts just once.

The proof of the last assertion is similar. Here use the relation
\[
 \frac{d}{dx} J_{k-1}(x)=\frac{k}{x} J_{k-1}(x)-J_{k}(x).
\]
(see \cite{GR} equation 8.472.2).
\end{proof}

\begin{lemma} \label{simple est}
Let $0<\delta<\varepsilon<1$ and $\phi_{\delta,\varepsilon}(x)=w_{\delta}(x)-w_{\varepsilon}(x)$. 
For $2<Y=o(X)$ as  $X \rightarrow \infty$, we have for any $\epsilon>0$ that
\[
\sum_{1 \le |n| < X} d^2(|n|) \, \Big(\mathcal B_d( \phi_{\delta,\varepsilon})\Big(\frac{n}{Y}\Big)\Big)^2
\ll  \varepsilon Y \log^3 Y +Y^{1/2+\epsilon}
\]
and
\[
\sum_{1 \le n < X} \rho_f(n)^2 \, \Big(\mathcal B_f( \phi_{\delta,\varepsilon})\Big(\frac{n}{Y}\Big)\Big)^2
\ll \varepsilon Y+Y^{3/5},
\]
where the implied constant depends on $f$.
\end{lemma}

\begin{proof}

For the first assertion we will only consider the sum over  $n>0$ since the terms with 
$n<0$ are handled in the same way.
We cite the formula 
\begin{equation}
\sum_{n\leq t}d^{2}\left(n\right)=c_3 t\log^{3}t
+c_2 t \log^2 t+c_1 t \log t +c_0t
+O(t^{1/2+\epsilon}),
\label{eq:varformula}
\end{equation}
where $c_3,c_2, c_1$ and $c_0$ are absolute constants
with $c_3=1/\pi^2$, (see \cite{Ramanujan} and equation 14.30 of \cite{Ivic 2}).  Let $Q_3(x)=\sum_{j=0}^3 c_j x^j$. We have
\begin{equation} \label{sum by parts}
\begin{split}
\sum_{1 \le n < X} d^2(n) \,\Big(\mathcal B_{d}( \phi_{\delta,\varepsilon})\Big(\frac{n}{Y}\Big)\Big)^2
=& \int_{1/2}^{X} \! \Big(\mathcal B_{d}( \phi_{\delta,\varepsilon})\Big(\frac{x}{Y}\Big)\Big)^2(Q_3(\log x)+Q_3'(\log x)) dx\\
&+\int_{1/2}^{X} \! \Big(\mathcal B_{d}( \phi_{\delta,\varepsilon})\Big(\frac{x}{Y}\Big)\Big)^2 dR(x)\\
=&I_1+I_2,
\end{split}
\end{equation}
where $R(x) \ll x^{1/2+\epsilon}$.

First we consider $I_2$. Integrating by parts we get
\[
I_2 =R(x) \Big(\mathcal B_{d}( \phi_{\delta,\varepsilon})\Big(\frac{x}{Y}\Big)\Big)^2 \Big|_{1/2}^{X}
-\frac{2}{Y}\int_{1/2}^{X} R(x) \mathcal B_{d}( \phi_{\delta,\varepsilon})\Big(\frac{x}{Y}\Big)
(\mathcal B_{d}( \phi_{\delta,\varepsilon}))'\Big(\frac{x}{Y}\Big) \, dx.
\]
Applying Lemma \ref{Bessel} and Lemma \ref{deriv est} we see that 
\begin{equation} \label{I2 bd}
I_2 \ll Y^{1/2+\epsilon}+Y^{1/4}\log{Y}
\int_{1/2}^Y x^{-3/4+\epsilon} \, dx
+ Y \int_{Y}^X x^{-3/2+\epsilon} \,dx 
\ll Y^{1/2+2\epsilon}.
\end{equation}
Next, observe that
\begin{align*}
I_1 \ll& Y  \log^3 Y \int_{0}^{\infty} \Big(\mathcal B_{d}( \phi_{\delta,\varepsilon})(x)\Big)^2 dx+Y\Big(\int_{1/(2Y)}^{Y^{2}}\! \!+ \! \!
\int_{Y^{2}}^{\infty}\Big) \Big(\mathcal B_d(\phi_{\delta,\varepsilon}) (x)\Big)^2 \! \, Q_3(|\log x|) \, dx \\
\ll & Y \log^3 Y \int_{0}^{\infty} \Big(\mathcal B_{d}( \phi_{\delta,\varepsilon})(x)\Big)^2 dx+Y^{\epsilon},
\end{align*}
by Lemma \ref{Bessel}.
From Proposition 2.3 of ~\cite{FGKM} it follows that
$\rVert \mathcal B_{d}( \phi_{\delta,\varepsilon}) \rVert_{L^2}^2=\rVert \mathcal \phi_{\delta,\varepsilon} \rVert_{L^2}^2$, where the $L^2$ norm
is computed in $\mathbb R^{\times}$ with respect
to Lebesgue measure. Also,
$\lVert  \phi_{\delta,\varepsilon} \rVert_{L^2}^2 \ll \varepsilon$. Thus, 
\[
I_1 \ll \varepsilon \, Y  \log^3 Y+Y^\epsilon.
\]
The proof of the first assertion follows
by this \eqref{sum by parts} and \eqref{I2 bd}.

To prove the other assertion we argue similarly. First, we cite the formula Rankin and Selberg (see \cite{Rankin} and \cite{Selberg})
\begin{equation} \label{Rankin Selberg}
\sum_{1 \le n <X} \rho_f^2(n)
=c_fX+O_{f}(X^{3/5}).
\end{equation}
We have
\[
\sum_{1 \le n < X} \rho_f(n)^2 \, \Big(\mathcal B_f( \phi_{\delta,\varepsilon})\Big(\frac{n}{Y}\Big)\Big)^2=
c_f \, \int_{1/2}^{X} \! \Big(\mathcal B_f( \phi_{\delta,\varepsilon})\Big(\frac{x}{Y}\Big)\Big)^2dx
+\int_{1/2}^{X} \! \Big(\mathcal B_{f}( \phi_{\delta,\varepsilon})\Big(\frac{x}{Y}\Big)\Big)^2 d\mathcal R(x),
\]
where $\mathcal R(x) \ll x^{3/5}$.
The first integral is $\ll \varepsilon Y$.
Integrating by parts, then applying Lemmas \ref{Bessel} and \ref{deriv est} we see that the second integral is $\ll Y^{3/5}$.

\end{proof}

\subsection{The proof of Theorems \ref{sharp divisor}
and \ref{sharp cusp forms}
}

We first deduce Theorems \ref{sharp divisor}
and \ref{sharp cusp forms} from the
following two lemmas.

\begin{lemma} \label{smooth to smooth}
Let $0<\varepsilon<1$ be fixed. Suppose that $X^{1/2}<p<X^{1-\theta}$ for some $\theta>0$.
For $\delta$ such that $(p/X)^{3/4-\eta} \ll \delta < \varepsilon$, 
for some $\eta>0$, we have for $\star=f$ or $d$ 
that
\[
\frac{1}{p-1} \sum_{1 \le a < p}
\bigg( E_{\star}(X,p,a; w_{\delta})-E_{\star}(X,p,a; w_{\epsilon}) \bigg)^2 =O(\varepsilon).
\]
\end{lemma}

\begin{lemma} \label{smooth to sharp}
Suppose that $p<X^{1-\theta}$ for some $\theta>0$.
If $2p/X \le \delta \ll (p/X)^{1/2} (pX)^{-\eta}$, for some
$\eta>0$, we have for $\star=f$ or $d$ that
\[
\frac{1}{p-1} \sum_{1 \le a < p}
\bigg| E_{\star}(X,p,a)-E_{\star}(X,p,a;w_{\delta}) \bigg| \ll  p^{-\eta}.
\]
\end{lemma}

\begin{proof}[Proof of Theorems
\ref{sharp divisor} and \ref{sharp cusp forms}]
Let $\varepsilon>0$ be fixed.
Also, let $X=p^2/\Phi(p)$ where $\Phi(p)$ tends
to infinity with $p$ in a way such that $\log \Phi(p)
=o(\log p)$. Corollary 1.4
of ~\cite{FGKM} gives as $p \rightarrow \infty$ that
\begin{equation} \label{prog thm}
 \frac{1}{p-1} \# \bigg\{ a \in \mathbb F_p^{\times} : \alpha \le   E_{\star}(X,p,a; w_{\varepsilon}) 
\le \beta\bigg\}=\frac{1}{\sqrt{2\pi}} \int_{\alpha}^{\beta} \! e^{-x^2/2} \, dx+o(1),
\end{equation}
for $\star=d$ or $f$.

Take $\delta=p^{-2/3}$.
By Lemma
\ref{smooth to smooth}
we have
\[
\frac{1}{p-1}\# \bigg\{ a \in \mathbb F_p^{\times} :
\bigg|E_{\star}(X,p,a;w_{\delta})-E_{\star}(X,p,a;w_{\varepsilon})\bigg| > \varepsilon^{1/3} \bigg\}
\ll \varepsilon^{1/3}.
\]
Since, $\varepsilon>0$ is fixed we see that by \eqref{prog thm}
\[
\frac{1}{p-1} \# \bigg\{ a \in \mathbb F_p^{\times} : \alpha \le   E_{\star}(X,p,a; w_{\delta}) 
\le \beta\bigg\}=\frac{1}{\sqrt{2\pi}} \int_{\alpha}^{\beta} \! e^{-x^2/2} \, dx+O(\varepsilon^{1/3})+o(1).
\]
Additionally, by Lemma \ref{smooth to sharp} we have
\[
\frac{1}{p-1}\# \bigg\{ a \in \mathbb F_p^{\times} :
\bigg|E_{\star}(X,p,a)-E_{\star}(X,p,a;w_{\delta})\bigg| > \varepsilon^{1/3} \bigg\}
\ll \varepsilon^{-1/3} p^{-1/18}. 
\]
Thus,
\[
\frac{1}{p-1} \# \bigg\{ a \in \mathbb F_p^{\times} : \alpha \le   E_{\star}(X,p,a) 
\le \beta\bigg\}=\frac{1}{\sqrt{2\pi}} \int_{\alpha}^{\beta} \! e^{-x^2/2} \, dx+O(\varepsilon^{1/3})+O(\varepsilon^{-1/3} p^{-1/18})+o(1).
\]
Since, $\varepsilon>0$ is arbitrary the result follows.
\end{proof}

For integers $a,b,$ and $c \ge 1$ the Kloosterman $S(a,b;c)$
is given by
\[
S(a,b;c)=\sum_{\substack{x \pamod c \\ (x,c)=1 }}
e\Big( \frac{xa+\bar x b}{c}\Big),
\]
where $x \bar x \equiv 1 \pmod c$ and $e(x)=e^{2\pi i x}$.
Let
\[
\tmop{Kl}_2(a,b;c)=\frac{S(a,b;c)}{c^{1/2}}.
\]

Before proving Lemma \ref{smooth to smooth}
we require a version of the Voronoi summation formula
that is proved in ~\cite{FGKM}.
\begin{lemma}[Proposition 2.1 of ~\cite{FGKM}] \label{voronoi}
Let $Y=p^2/X$. Then for any non negative smooth function $w$
compactly supported in the positive reals we have
\[
E_{\star}(X,p,a; w)=\frac{1}{ \sigma_{\star}(Y)} \sum_{n \neq 0} \tau_{\star}(n) \mathcal B_{\star}( w ) \Big( \frac{n}{Y}\Big) \tmop{Kl}_2(a,n;p),
\]
where 
\[
\sigma_{\star}^2 (Y)=
\begin{cases}
\frac{2}{\pi^2} Y \log^3 Y  \quad \mbox{ if  } \quad \star=d,\\
c_f Y \quad \mbox{ if } \quad \star=f,
\end{cases}
\]
$\tau_{d}(n)=d(|n|)$ and
\[
\tau_{f}(n)=
\begin{cases}
\rho_f(n) \quad \mbox{ if } \quad  n>0, \\
0 \quad \mbox{ if } \quad  n < 0.
\end{cases}
\]
\end{lemma}

In the proof of Lemma \ref{smooth to smooth}
we will need to estimate
\[
\frac{1}{p-1}\sum_{1 \le a < p} \tmop{Kl}_2(a,m;p)\tmop{Kl}_2(a,n;p).
\]
First note that if $m$ or $n$ is divisible by $p$
it is easily seen that this
is $\ll 1/p$. Next, for integers $\ell$ define
\[
f(\ell)=
\begin{cases}
0 \quad \mbox{  if  }  \quad \ell \equiv 0 \pamod p, \\
e(\overline \ell/p) \quad \mbox{ otherwise.} 
\end{cases}
\]
Applying, the discrete Plancherel formula we have
for $m$ and $n$ not divisible by $p$ that 
\begin{equation} \notag
\begin{split}
\frac{1}{p-1}\sum_{0 \le a <p}
\tmop{Kl}_2(a,m;p)\tmop{Kl}_2(a,n;p)
=& \frac{1}{p-1} \sum_{0 \le b <p} f(mb)\overline{f(nb)} \\
=& \frac{1}{p-1}\sum_{1 \le b <p} e\Big(\frac{\overline{b}(\overline m-\overline n)}{p} \Big).
\end{split}
\end{equation}
We conclude that
\begin{equation} \label{kloosterman}
\frac{1}{p-1}\sum_{1 \le a <p}
\tmop{Kl}_2(a,n;p) \tmop{Kl}_2(a,m;p)
=
\begin{cases}
1+O\Big(\frac{1}{p^2} \Big)  \mbox{ if } n \equiv m \! \! \! \pmod p  \mbox{ and } p \nmid n,  \\
O\Big(\frac{1}{p}\Big) \mbox{ otherwise.}
\end{cases}
\end{equation}

\begin{proof}[Proof of Lemma \ref{smooth to smooth}]
Let $\phi_{\delta,\varepsilon}(x)=w_{\delta}(x)-w_{\varepsilon}(x)$ and $Y=p^2/X$. Applying Lemma \ref{voronoi} we see that
\begin{equation*}
\begin{split}
&\frac{1}{p-1} \sum_{1 \le a < p}
\bigg( E_{\star}(X,p,a; w_{\delta})-E_{\star}(X,p,a;w_{\varepsilon}) \bigg)^2
 \\
& \qquad \qquad \qquad \ll \frac{1}{p \cdot \sigma_{\star}^2(Y)}
\sum_{1 \le a < p} \bigg(
\sum_{1 \le |n| \le p/2}
\tau_{\star}(n)
\mathcal B_{\star}(\phi_{\delta,\varepsilon})\Big(\frac{n}{Y} \Big)\tmop{Kl}_2(a,n;p) \bigg)^2 \\
&  \qquad \qquad \qquad \quad
+\frac{1}{p \cdot \sigma_{\star}^2(Y)}
\sum_{1 \le a < p} \bigg(
\sum_{ |n| > p/2}
\tau_{\star}(n)
 \mathcal B_{\star}(\phi_{\delta,\varepsilon})\Big(\frac{n}{Y} \Big)\tmop{Kl}_2(a,n;p) \bigg)^2=\Sigma_1+\Sigma_2.
\end{split}
\end{equation*}

We first consider $\Sigma_2$. We have
\[
\Sigma_2
=
\frac{1}{  \sigma_{\star}^2(Y)}
\sum_{|m|,|n| > p/2 } \tau_{\star}(m)\tau_{\star}(n) 
\mathcal B_{\star}(\phi_{\delta,\varepsilon})\Big(\frac{m}{Y}\Big)
\mathcal B_{\star}(\phi_{\delta,\varepsilon})\Big(\frac{n}{Y}\Big) 
\frac{1}{p} \sum_{1 \le a < p} \tmop{Kl}_2(a,m;p)\tmop{Kl}_2(a,n;p).
\]
By \eqref{kloosterman}, Lemma \ref{Bessel},
and the bound $\tau_{\star}(n) \ll n^\varepsilon$ the contribution of the terms with $m \equiv n \pmod p$ to $\Sigma_2$
is
\begin{equation*}
\begin{split}
&\ll \frac{1}{ \sigma_{\star}^2(Y)}
\sum_{\substack{|m|,|n| > p/2 \\ m \equiv n \pamod p}} \bigg| \tau_{\star}(m)\tau_{\star}(n) 
\mathcal B_{\star}(\phi_{\delta,\varepsilon})\Big(\frac{m}{Y}\Big)
\mathcal B_{\star}(\phi_{\delta,\varepsilon})\Big(\frac{n}{Y}\Big) \bigg|\\
& 
\ll 
\frac{Y^{5/2}}{\sigma_{\star}^2(Y) \delta^2}
\sum_{|m| > p/2} \frac{1}{|m|^{5/4-\varepsilon}}
\sum_{\substack{|n| > p/2 \\ n \equiv m \pamod p}}
\frac{1}{|n|^{5/4-\varepsilon}} \ll
\frac{Y^{5/2}}{p^{3/2-\varepsilon} \sigma_{\star}^2(Y) \delta^2 }.
\end{split}
\end{equation*} 
Similarly, applying \eqref{kloosterman} and Lemma \ref{Bessel}
the sum of the remaining terms in $\Sigma_2$ is
\[
\ll \frac{1}{p \cdot \sigma_{\star}^2(Y)}\bigg(\sum_{ |n|  > p/2} 
\bigg| \tau_{\star}(n) \mathcal B_{\star}(\phi_{\delta,\varepsilon})\Big(\frac{n}{Y} \Big) \bigg|\bigg)^2 
\ll \frac{Y^{5/2}}{p \cdot \sigma_{\star}^2(Y) \delta^2}
\bigg(\sum_{|n| > p/2 } 
\frac{1}{|n|^{5/4-\varepsilon}}  \bigg)^2 \ll \frac{Y^{5/2}}{p^{3/2-\varepsilon} \sigma_{\star}^2(Y) \delta^2 }.
\]
Thus,
\begin{equation}\label{sigma2 bd}
\Sigma_2 \ll \frac{Y^{3/2}}{p^{3/2-\varepsilon}  \delta^2 }= \frac{p^{3/2+\varepsilon}}{X^{3/2} \delta^2}.
\end{equation}

The estimation of $\Sigma_1$ is similar. First note that
\begin{equation} \label{eqn tired}
\begin{split}
\Sigma_1
& \ll \frac{1}{p \cdot \sigma_{\star}^2(Y)}
\sum_{1 \le a < p} \bigg(
\sum_{1 \le n \le p/2}
\tau_{\star}(n)
\mathcal B_{\star}(\phi_{\delta,\varepsilon})\Big(\frac{n}{Y} \Big)\tmop{Kl}_2(a,n;p) \bigg)^2 \\
& \quad
+\frac{1}{p \cdot \sigma_{\star}^2(Y)}
\sum_{1 \le a < p} \bigg(
\sum_{ -p/2 \le n \le -1}
\tau_{\star}(n)
 \mathcal B_{\star}(\phi_{\delta,\varepsilon})\Big(\frac{n}{Y} \Big)\tmop{Kl}_2(a,n;p) \bigg)^2.
\end{split}
\end{equation}
Here if $m \equiv n \pmod p$ then $m=n$
so that the first sum above is
\begin{equation} \notag
\ll
\frac{1}{ \sigma_{\star}^2(Y)}
\sum_{1 \le n \le p/2 } \tau_{\star}(n)^2 
 \Big(\mathcal B_{\star}(\phi_{\delta,\varepsilon})\Big(\frac{n}{Y}\Big)\Big)^2
+\frac{1}{p \cdot \sigma_{\star}^2(Y)}\bigg(\sum_{1 \le n  \le p/2} \bigg|\tau_{\star}(n)\mathcal B_{\star}(\phi_{\delta,\varepsilon})\Big(\frac{n}{Y} \Big) \bigg| \bigg)^2.
\end{equation}
Now, apply the Cauchy-Schwarz inequality to the second sum above. Thus, by Lemma \ref{simple est} the first sum in \eqref{eqn tired}
is
\[
\ll
\frac{1}{ \sigma_{\star}^2(Y)} \cdot \sum_{1 \le n < p/2 } \tau_{\star}(n)^2 
 \Big(\mathcal B_{\star}(\phi_{\delta,\varepsilon})\Big(\frac{n}{Y}\Big)\Big)^2 \ll \varepsilon.
\]
The second sum in \eqref{eqn tired} satisfies this bound
as well. Thus, by this and \eqref{sigma2 bd} the proof is completed.

\end{proof}

\begin{proof}[Proof of Lemma \ref{smooth to sharp}]
Observe that for any $\varepsilon>0$
\[
\mathcal S_{\star}(X,p,a)-\mathcal S_{\star}(X,p,a;w_{\delta})=
\sum_{\substack{n \le X \\ n \equiv a \pamod p} }\tau_{\star}(n)\Big(1-w_{\delta}\Big(\frac{n}{X}\Big)\Big)
\ll \frac{\delta X^{1+\varepsilon}}{p},
\]
uniformly in $a$.
Similarly,
\[
\frac1p \sum_{n \le X} \tau_{\star}(n)\Big(1-w_{\delta}\Big(\frac{n}{X}\Big)\Big)
\ll \frac{\delta X^{1+\varepsilon}}{p},
\]
and
\[
\frac{1}{p^2}\int_0^X (\log y+2\gamma-2\log p)\Big(1-w_{\delta}\Big(\frac{y}{X}\Big) \Big)\, dy
\ll \frac{\delta X }{p^2} \log( Xp).
\]
Thus,
\[
\mathcal M_{\star}(X,p,a)-\mathcal M_{\star}(X,p,a;w_{\delta}) \ll \frac{\delta X^{1+\varepsilon}}{p}.
\]
It follows that, uniformly in $a$, we have
\[
\bigg| E_{\star}(X,p,a)-E_{\star}(X,p,a;w_{\delta}) \bigg|
\ll \frac{\delta X^{1/2+\varepsilon}}{ p^{1/2}}.
\]
We conclude that
\[
\frac{1}{p-1} \sum_{1 \le a < p}
\bigg| E_{\star}(X,p,a)-E_{\star}(X,p,a;w_{\delta}) \bigg| \ll  p^{-\varepsilon}.
\]
for $2p/X \le \delta \ll (p/X)^{1/2} (pX)^{-\varepsilon}$.

\end{proof}

\section{The distribution of $d(n)$ in short intervals}

\subsection{The Variance of $S(x,L)$}

To take averages, instead of working first with Lebesgue measure,
we use a smooth average around $T$, so we take a Schwartz function
$\omega\geq0$ supported on the positive real numbers with
a unit mass. Define our averages by
\[
\left\langle f \right\rangle =\int_{-\infty}^{\infty}f\left(x\right)\omega\left(\frac{x}{T}\right)\frac{\mbox{d}x}{T}=\int_{-\infty}^{\infty}f\left(xT\right)\omega\left(x\right)\mbox{d}x.
\]
Also, for $f \in L^1(\mathbb R)$ let
\[
\widetilde f(\xi)=\int_{-\infty}^{\infty}f\left(x\right) e^{-2\pi i  \xi \sqrt{|x|}}\mbox{d}x.
\]
By repeatedly integrating by parts it follows for any $A  \ge 1$ that
\begin{equation} \label{trans bd} 
\widetilde \omega(\xi) \ll  \min \Big\{1, |\xi|^{-A} \Big\},
\end{equation}
where the implied constant depends on $\omega$.

Define $F\left(x\right)=x^{-1/4}\Delta\left(x\right).$
Let $L=L\left(T\right)\to\infty$ as $T\to\infty$,
with the condition that $L\ll T^\varepsilon$ for all $\varepsilon>0$. Define
\[
S\left(x,L\right)=F\left(\Big(\sqrt{x}+\frac{1}{L}\Big)^2\right)-F\left(x\right).
\]
Also, let
\[
\sigma^2=\frac{16}{\pi^2} \frac{\log^3 L}{L}.
\]

We first show that the expectation of $S\left(x,L\right)$ tends to
zero as $T\to\infty$:
\begin{lemma}
\textup{For all $\varepsilon>0$, $\left\langle S\left(x,L\right)\right\rangle =O\left(T^{-1/4+\varepsilon}\right)$.}\end{lemma}
\begin{proof}
We use the following formula (see Titchmarsh, \cite{Titchmarsh} (12.4.4)
for example)
\[
F\left(x\right)=\frac{1}{\pi\sqrt{2}}\sum_{n\leq T}\frac{d\left(n\right)}{n^{3/4}}\cos\left(4\pi \sqrt{ nx}-\frac{\pi}{4}\right)+O\left(T^{-1/4+\varepsilon}\right)
\]
uniformly for $T\leq x\leq2T$.

We conclude that for such $x$
\begin{align*}
S\left(x,L\right) & =\frac{-2}{\pi\sqrt{2}}\sum_{n\leq T}\frac{d\left(n\right)}{n^{3/4}}\sin\left(\frac{2\pi\sqrt{n}}{L}\right)\sin\left(4\pi\sqrt{n}\left(\sqrt{x}+\frac{1}{2L}\right)-\frac{\pi}{4}\right)\\
 & \quad +O\left(T^{-1/4+\varepsilon}\right).
\end{align*}
So
\begin{align*}
\left\langle S\left(x,L\right)\right\rangle  & =\frac{-2}{\pi\sqrt{2}}\sum_{n\leq T}\frac{d\left(n\right)}{n^{3/4}}\sin\left(\frac{2\pi\sqrt{n}}{L}\right)\left\langle \sin\left(4\pi\sqrt{n}\left(\sqrt{x}+\frac{1}{2L}\right)-\frac{\pi}{4}\right)\right\rangle \\
 & \quad +O\left(T^{-1/4+\varepsilon}\right).
\end{align*}
Note that
\begin{align*}
\left\langle \sin\left(4\pi\sqrt{n}\left(\sqrt{x}+\frac{1}{2L}\right)-\frac{\pi}{4}\right)\right\rangle  & =\frac{1}{2i}\left(\widetilde {\omega}\left(-2\sqrt{T n}\right)e^{\pi i\left(\frac{2\sqrt{n}}{L}-\frac{1}{4}\right)}\right.\\
 & \left.-\widetilde {\omega}\left(2\sqrt{Tn}\right)e^{-\pi i\left(\frac{2\sqrt{n}}{L}-\frac{1}{4}\right)}\right),
\end{align*}
so from the rapid decay of $\widetilde {\omega}$ we get that for all $A>0$
\[
\left\langle S\left(x,L\right)\right\rangle \ll\sum_{n=1}^{\infty}\frac{d\left(n\right)}{n^{3/4}}\frac{1}{T^{A/2}n^{A/2}}+O\left(T^{-1/4+\varepsilon}\right)=O\left(T^{-1/4+\varepsilon}\right).
\]
\end{proof}

We now compute the variance of $S\left(x,L\right):$
\begin{lemma} \label{variance}
\textup{$\left\langle S\left(x,L\right)^{2}\right\rangle \sim\frac{16}{\pi^{2}}\frac{\log^{3}L}{L}=\sigma^2$.}\end{lemma}
\begin{proof}
Again, from the formula (12.4.4) in \cite{Titchmarsh}, we get that for any (small) $\delta>0$
\begin{align}
S\left(x,L\right) & =\frac{-2}{\pi\sqrt{2}}\sum_{n\leq T^{1-\delta}}\frac{d\left(n\right)}{n^{3/4}}\sin\left(\frac{2\pi\sqrt{n}}{L}\right)\sin\left(4\pi\sqrt{n}\left(\sqrt x+\frac{1}{2L}\right)-\frac{\pi}{4}\right)\label{eq:difference}\\
 & \quad +O\left(T^{-1/4+\delta}\right).\nonumber
\end{align}
Denote the sum in (\ref{eq:difference}) by $P$.
By the Cauchy-Schwarz inequality
\[
\left\langle S\left(x,L\right)^{2}\right\rangle =\left\langle P^{2}\right\rangle +O\left(T^{-1/4+\delta}\sqrt{\left\langle P^{2}\right\rangle }+T^{-1/2+2\delta}\right)
\]
so (since $L\ll T^\varepsilon$ for all $\varepsilon>0$) it is enough to show that $\left\langle P^{2}\right\rangle \sim\frac{16}{\pi^{2}}\frac{\log^{3}L}{L}.$

Indeed, we first see that the contribution of the off-diagonal terms
is minor: their contribution equals
\begin{align}
 & \frac{2}{\pi^{2}}\sum_{n\neq m\leq T^{1-\delta}}\frac{d\left(n\right)d\left(m\right)}{\left(nm\right)^{3/4}}\sin\left(\frac{2\pi\sqrt{n}}{L}\right)\sin\left(\frac{2\pi\sqrt{m}}{L}\right)\label{eq:offdiag}\\
 & \qquad \qquad \qquad \qquad \times\left\langle \sin\left(4\pi\sqrt{n}\left(\sqrt x+\frac{1}{2L}\right)-\frac{\pi}{4}\right)\sin\left(4\pi\sqrt{m}\left(\sqrt x+\frac{1}{2L}\right)-\frac{\pi}{4}\right)\right\rangle \nonumber.
\end{align}
Observe that
\begin{gather*}
\sin\left(4\pi\sqrt{n}\left(\sqrt x+\frac{1}{2L}\right)-\frac{\pi}{4}\right)\sin\left(4\pi\sqrt{m}\left(\sqrt x+\frac{1}{2L}\right)-\frac{\pi}{4}\right)\\
=-\frac{1}{4}\left(e^{\left(4\pi\left(\sqrt x+\frac{1}{2L}\right)\left(\sqrt{n}+\sqrt{m}\right)-\frac{\pi}{2}\right)i}+e^{\left(4\pi\left(\sqrt x+\frac{1}{2L}\right)\left(-\sqrt{n}-\sqrt{m}\right)+\frac{\pi}{2}\right)i}\right.\\
\left.-e^{4\pi\left(\sqrt x+\frac{1}{2L}\right)\left(\sqrt{n}-\sqrt{m}\right)i}-e^{4\pi\left(\sqrt x+\frac{1}{2L}\right)\left(-\sqrt{n}+\sqrt{m}\right)i}\right).
\end{gather*}
So 
\begin{gather*}
\left\langle \sin\left(4\pi\sqrt{n}\left(\sqrt x+\frac{1}{2L}\right)-\frac{\pi}{4}\right)\sin\left(4\pi\sqrt{m}\left(\sqrt x+\frac{1}{2L}\right)-\frac{\pi}{4}\right)\right\rangle \\
=-\frac{1}{4}\left(\widetilde{\omega}\left(2\sqrt{T}\left(-\sqrt{n}-\sqrt{m}\right)\right)e^{\pi i\left(-\frac{1}{2}+\frac{2}{L}\left(\sqrt{n}+\sqrt{m}\right)\right)}\right.\\
\left.+\widetilde{\omega}\left(2\sqrt{T}\left(\sqrt{n}+\sqrt{m}\right)\right)e^{\pi i\left(\frac{1}{2}+\frac{2}{L}\left(-\sqrt{n}-\sqrt{m}\right)\right)}\right.\\
\left.-\widetilde{\omega}\left(2\sqrt{T}\left(-\sqrt{n}+\sqrt{m}\right)\right)e^{\frac{2\pi i}{L}\left(\sqrt{n}-\sqrt{m}\right)}\right.\\
\left.-\widetilde{\omega}\left(2\sqrt{T}\left(\sqrt{n}-\sqrt{m}\right)\right)e^{\frac{2\pi i}{L}\left(-\sqrt{n}+\sqrt{m}\right)}\right).
\end{gather*}
Since $\sqrt{m}\pm\sqrt{n}\gg T^{-1/2+\delta/2}$
and $\widetilde{\omega}$ rapidly decays by \eqref{trans bd}, we conclude that for every
$A>0$
\begin{gather*}
\left\langle \sin\left(4\pi\sqrt{n}\left(\sqrt{x}+\frac{1}{2L}\right)-\frac{\pi}{4}\right)\sin\left(4\pi\sqrt{m}\left(\sqrt x+\frac{1}{2L}\right)-\frac{\pi}{4}\right)\right\rangle \\
\ll\left(\sqrt{T}\left(\sqrt{m}\pm\sqrt{n}\right)\right)^{-A}\ll T^{-\delta A/2}.
\end{gather*}
So the sum in (\ref{eq:offdiag}) is bounded above by
\[
\sum_{n\neq m\leq T^{1-\delta}}T^{-\delta A/2}\leq T^{2-\delta A/2}.
\]

We get that for all $B>0$
\begin{align*}
\left\langle P^{2}\right\rangle  & =\frac{2}{\pi^{2}}\sum_{n\leq T^{1-\delta}}\frac{d^{2}\left(n\right)}{n^{3/2}}\sin^{2}\left(\frac{2\pi\sqrt{n}}{L}\right)\left\langle \sin^{2}\left(4\pi\sqrt{n}\left(\sqrt x+\frac{1}{2L}\right)-\frac{\pi}{4}\right)\right\rangle \\
 & \quad +O\left(T^{-B}\right).
\end{align*}
Note that
\begin{gather*}
\left\langle \sin^{2}\left(4\pi\sqrt{n}\left(\sqrt x+\frac{1}{2L}\right)-\frac{\pi}{4}\right)\right\rangle =\\
-\frac{1}{4}\left(\widetilde{\omega}\left(-4\sqrt{ T n}\right)e^{\pi i\left(-\frac{1}{2}+\frac{4}{L}\sqrt{n}\right)}+\widetilde{\omega}\left(4\sqrt{Tn}\right)e^{\pi i\left(\frac{1}{2}-\frac{4}{L}\sqrt{n}\right)}-2\right)\\
=\frac{1}{2}+O\left(T^{-B}\right).
\end{gather*}
So actually
\[
\left\langle P^{2}\right\rangle =\frac{1}{\pi^{2}}\sum_{n\leq T^{1-\delta}}\frac{d^{2}\left(n\right)}{n^{3/2}}\sin^{2}\left(\frac{2\pi\sqrt{n}}{L}\right)+O\left(T^{-B}\right).
\]
To evaluate the main term, write
\[
\sigma_{T^{1-\delta}}^{2}:=\frac{1}{\pi^{2}}\sum_{n\leq T^{1-\delta}}\frac{d^{2}\left(n\right)}{n^{3/2}}\sin^{2}\left(\frac{2\pi\sqrt{n}}{L}\right).
\]
Applying partial summation and using (\ref{eq:varformula}), we get
that
\begin{align*}
\sigma_{T^{1-\delta}}^{2} & \sim\frac{1}{\pi^{4}}\int_{1}^{T^{1-\delta}}\frac{\log^{3}x}{x^{3/2}}\sin^{2}\left(\frac{2\pi\sqrt{x}}{L}\right)\mbox{d}x\\
 & =\frac{16}{\pi^{4}L}\int_{1/L}^{T^{1/2-\delta/2}/L}\frac{\log^{3}\left(Ly\right)}{y^{2}}\sin^{2}\left(2\pi y\right)\mbox{d}y\\
 & \sim\frac{16}{\pi^{4}}\frac{\log^{3}L}{L}\int_{1/L}^{T^{1/2-\delta/2}/L}\frac{\sin^{2}\left(2\pi y\right)}{y^{2}}\mbox{d}y\\
 & \sim\frac{16}{\pi^{4}}\frac{\log^{3}L}{L}\int_{0}^{\infty}\frac{\sin^{2}\left(2\pi y\right)}{y^{2}}\mbox{d}y,
\end{align*}
where the last relation holds because
\[
\int_{0}^{1/L}\frac{\sin^{2}\left(2\pi y\right)}{y^{2}}\mbox{d}y\ll\frac{1}{L}
\]
and
\[
\int_{T^{1/2-\delta/2}/L}^{\infty}\frac{\sin^{2}\left(2\pi y\right)}{y^{2}}\mbox{d}y\ll\frac{L}{T^{1/2-\delta/2}}.
\]
Since $\int_{0}^{\infty}\frac{\sin^{2}\left(2\pi y\right)}{y^{2}}\mbox{d}y=\pi^{2},$
we conclude that
\[
\left\langle P^{2}\right\rangle \sim\sigma_{T^{1-\delta}}^{2}\sim\frac{16}{\pi^{2}}\frac{\log^{3}L}{L}=\sigma^2.
\]

\end{proof}
For any $M=O\left(T^{1-\delta}\right)$ such that $L/\sqrt{M}\to0$,
define the ``short'' sum
\[
S\left(x,L,M\right)=\frac{-2}{\pi\sqrt{2}}\sum_{n\leq M}\frac{d\left(n\right)}{n^{3/4}}\sin\left(\frac{2\pi\sqrt{n}}{L}\right) \sin\left(4\pi\sqrt{n}\left(\sqrt{x}+\frac{1}{2L}\right)-\frac{\pi}{4}\right).
\]
By the same arguments as above, we conclude that
\[
\left\langle S\left(x,L,M\right)\right\rangle =O\left(T^{-B}\right)
\]
for all $B>0$, with the implied constant independent of $M$, and
\[
\left\langle S\left(x,L,M\right)^{2}\right\rangle \sim\sigma^2.
\]
The advantage of working with $S\left(x,L,M\right)$ is that we can also calculate its higher moments -- this will be done in the next section.

By considerations similar to the above, we get that the normalized distance between
the short and the long sums tends to zero in the $L^{2}$-norm:
\begin{lemma} \label{conv prob}
As $T\to\infty$
\[
\left\langle \left(\frac{S\left(x,L\right)-S\left(x,L,M\right)}{\sigma}\right)^{2}\right\rangle =o(1).
\]
\end{lemma}
\begin{proof}
We have for any (small) $\delta>0$
\begin{gather}
S\left(x,L\right)-S\left(x,L,M\right)=\label{eq:longminusshort}\\
\frac{-2}{\pi\sqrt{2}}\sum_{M<n\leq T^{1-\delta}}\frac{d\left(n\right)}{n^{3/4}}\sin\left(\frac{2\pi\sqrt{n}}{L}\right)\sin\left(4\pi\sqrt{n}\left(\sqrt x+\frac{1}{2L}\right)-\frac{\pi}{4}\right)\nonumber \\
+O\left(T^{-1/4+\delta}\right).\nonumber
\end{gather}
Denoting the sum in (\ref{eq:longminusshort}) by $P$, we get by
Cauchy-Schwarz's inequality that
\[
\left\langle \left(\frac{S\left(x,L\right)-S\left(x,L,M\right)}{\sigma}\right)^{2}\right\rangle =\frac{1}{\sigma^{2}}\left\langle P^{2}\right\rangle +O\left(\frac{T^{-1/4+\delta}}{\sigma^{2}}\sqrt{\left\langle P^{2}\right\rangle }+\frac{T^{-1/2+2\delta}}{\sigma^{2}}\right)
\]
so (since $L\ll T^\varepsilon$ for all $\varepsilon>0$) it is enough to show that $\left\langle P^{2}\right\rangle =o\left(\sigma^{2}\right)$.

Indeed, by the same reasoning as in the proof of Lemma \ref{variance}
\begin{equation} \label{sigma}
\left\langle P^{2}\right\rangle \sim\frac{16}{\pi^{4}}\frac{\log^{3}L}{L}\int_{\sqrt{M}/L}^{\infty}\frac{\sin^{2}\left(2\pi y\right)}{y^{2}}\mbox{d}y\ll\sigma^{2}\frac{L}{\sqrt{M}}=o\left(\sigma^{2}\right)
\end{equation}
since $L/\sqrt{M}\to0$ as $T\to\infty$.\end{proof}

\subsection{Higher moments}
Define
\[
\sigma_M^2=\frac{1}{\pi^{2}}\sum_{n\leq M}\frac{d^{2}\left(n\right)}{n^{3/2}}\sin^{2}\left(\frac{2\pi\sqrt{n}}{L}\right).
\]
In this section we calculate the moments of $ S(x, L, M)/\sigma_M$ and prove:

\begin{proposition} \label{moments prop}
Let $m \ge 2$ be an integer.
Suppose $M  \ll T^{(1-\epsilon)/(2^{m-1}-1)}$ for some $\epsilon>0$, $L$ tends to infinity with $T$, and $L/\sqrt M \rightarrow 0$. As $T \rightarrow \infty$ we have
\begin{equation}
\begin{split}
\left\langle \left( \frac{S(x, L, M)}{\sigma_M} \right)^m \right \rangle
=
\begin{cases}
 \frac{m!}{2^{m/2} (m/2)! }+o(1) \mbox{ if $m$ is even}, \\
o(1) \mbox{ if $m$ is odd}.
\end{cases}
\end{split}
\end{equation}
\end{proposition}

We first cite two auxiliary lemmas.

\begin{lemma}[Theorem 2 of \cite{B}] \label{lin indp}
Let $q$ denote a square free, positive integer. The set $\{ \sqrt{q} \}$ is linearly independent over $\mathbb Q$.
\end{lemma}

\begin{lemma}[Lemma 3.5 of \cite{HR}]
For $j=1, \ldots, m$, let $n_j \le M$ be positive integers and let $\epsilon_j \in \{-1,1\}$ be such that
\[
\sum_{j=1}^{m} \epsilon_j \sqrt{n_j} \neq 0.
\]
Then,
\[
\bigg|\sum_{j=1}^{m} \epsilon_j \sqrt{n_j} \, \bigg| \geq \frac{1}{(m \sqrt{M})^{2^{m-1}-1}}.
\]
\end{lemma}

Let $\{ X(q) \}_q$ be a sequence of independent random variables uniformly distributed on the unit circle,
where the index $q$ runs over the square-free numbers.

\begin{lemma} \label{mean values}
Let $a_n, b_n$ be complex numbers such that $a_n, b_n \ll 1$. Also, let $0<\epsilon<1/2$ and $k, \ell \geq 0$
be integers.
Suppose $M \leq T^{(1-\epsilon)/(2^{k+\ell-1}-1)}$. 
Write $n=qf^2$ where $q$ is the square-free part of $n$ and let $\{X(q)\}_q$ 
be a sequence of independent random variables uniformly distributed on the unit circle.
We have for any $A \ge 1$ that
\begin{equation} \notag
\begin{split}
&
\left\langle \bigg(\sum_{n \le M}
a_n e(2\sqrt{nx} ) \bigg)^{k}
\overline{\bigg(\sum_{n \le M} b_n e(2\sqrt{nx})\bigg)}^{\ell}  \right\rangle
\\
& \qquad \qquad \qquad \qquad
= \mathbb E \bigg( \Big(\sum_{qf^2 \le M}
a_{qf^2} (X(q))^f \bigg)^{k}
\overline{\Big(\sum_{qf^2 \le M} b_{qf^2} (X(q))^f\bigg)}^{\ell} \bigg)+O(T^{-A}),
\end{split}
\end{equation}
where the implied constant depends on $k, \ell, A,$
and $\epsilon$.
\end{lemma}

\begin{proof}
Integrating term-by-term gives
\begin{equation}
\begin{split}
\int_{\mathbb R}
&\bigg(\sum_{n \le M}
a_n e(2\sqrt{nx} ) \bigg)^{k}
\overline{\bigg(\sum_{n \le M} b_n e(2\sqrt{nx} )\bigg)}^{\ell} \omega\Big( \frac{t}{T}\Big) \, \frac{dt}{T}\\
& \qquad \qquad \qquad =
\sum_{\substack{n_1, \ldots, n_k \le M \\ m_1, \ldots, m_{\ell} \le M }}
\prod_{r=1}^k a_{n_r} \prod_{s=1}^{\ell} \overline{b_{m_s}} \cdot \widetilde \omega \Big(2\sqrt{T}\Big(
\sum_{s=1}^{\ell} \sqrt{m_s}
-\sum_{r=1}^k \sqrt{n_r}\Big)\Big)\\
& \qquad \qquad \qquad =
\Sigma_D+\Sigma_O,
\end{split}
\end{equation}
where $\Sigma_D$ contains the terms where $\sum_{s=1}^{\ell} \sqrt{m_s}
-\sum_{r=1}^k \sqrt{n_r}=0$ and $\Sigma_O$ consists of the remaining terms.

For non negative integers $a$ and $b$
\begin{equation} \notag
\mathbb E\left( X (q)^a \overline{X(q)}^b \right)=
\begin{cases}
  1 & \mbox{if  } a = b, \\
  0 & \mbox{if  } a \neq b.
\end{cases}
\end{equation}
Writing $n_r=q_rf_r^2$ where $q_r$ is square-free, we have
by Lemma \ref{lin indp} and the independence of the random variables $X(q_r)$ that
\[
\Sigma_D=\mathbb E \bigg( \Big(\sum_{qf^2 \le M}
a_{qf^2} (X(q))^f\bigg)^{k}
\overline{\Big(\sum_{qf^2 \le M} b_{qf^2} (X(q))^f\bigg)}^{\ell} \bigg).
\]

Next, note that by Lemma \ref{mean values}
for each term in $\Sigma_O$ we have
\[
\bigg|\sum_{s=1}^{\ell} \sqrt{m_s}
-\sum_{r=1}^k \sqrt{n_r}\bigg|\geq \frac{1}{((k+\ell) \sqrt{M})^{2^{k+\ell-1}-1}}.
\]
Since, by \eqref{trans bd} $\widetilde \omega$ 
decays rapidly
and that $a_n, b_n \ll 1$ we have
for any $N \ge 1$
\begin{equation} \notag
\begin{split}
\Sigma_O
\ll&
\sum_{\substack{n_1, \ldots, n_k \le M \\ m_1, \ldots, m_{\ell} \le M }}
\widetilde \omega \Big(2\sqrt{T}\Big(
\sum_{s=1}^{\ell} \sqrt{m_s}
-\sum_{r=1}^k \sqrt{n_r}\Big)\Big) \\
\ll &
\bigg(\frac{((k+\ell) \sqrt{M})^{2^{k+\ell-1}-1}}{\sqrt{T}} \bigg)^N M^{k+\ell} \ll T^{-A}
\end{split}
\end{equation}
by our assumption on $M$ and since $N$ is arbitrary. Here the implied constant depends on $k, \ell, A,$ and $\epsilon$.
\end{proof}

Let
\[
Y(q) =\frac{-2}{ q^{3/4} \pi \sqrt 2}\sum_{f \leq \sqrt{M/q}} \frac{d(qf^2)}{f^{3/2}} e\Big(\frac{f\sqrt{q}}{L}-\frac18\Big)
\sin\Big(2\pi  \frac{f\sqrt q}{L}  \Big) (X (q))^f.
\]

\begin{lemma} \label{Y bd} Suppose $m \geq 2$.
For any $\delta>0$ we have
\[
\mathbb E \Big((\tmop{Im} Y(q))^m\Big)  
\ll
\begin{cases}
\frac{1}{q^{3m/4-m\delta}} \quad \mbox{ if } \quad
q>L^2 \\
\frac{1}{q^{\frac{(m-1)}{2}} L^{\frac{m}{2}+1-2m\delta}} \quad \mbox{ if } \quad
q \le L^2,
\end{cases}
\]
where the implied constant depends on $\delta$ and $m$.
\end{lemma}
\begin{proof}
Applying the bound $d(n) \ll n^{\delta}$
for any $\delta>0$, we have
\[
Y(q) \ll \frac{1}{q^{3/4-\delta}}.
\]
The claimed estimate for $q>L^2$ follows.

We now assume $q \le L^2$. It suffices to bound the mixed moments
\[
\mathbb E \Big(Y(q)^k \overline{Y(q)}^{\ell} \Big),
\]
with $k+\ell=m$. Note that if either $k$ or $\ell$ equals zero then the expectation also equals zero.
Next, consider the case where both $k$ and $\ell$ are positive, we have
\[
\mathbb E \Big(Y(q)^k \overline{Y(q)}^{\ell} \Big)=
\frac{(-1)^m 2^{m/2}}{q^{3m/4} \pi^m}\sum_{\substack{f_1, \ldots, f_m \leq \sqrt{M/q}
\\ f_1+\cdots+f_k=f_{k+1}+\cdots+f_{m}}} \prod_{j=1}^m \frac{d(qf_j^2)}{f_j^{3/2}} e\Big(\frac{f_j\sqrt{q}}{L}-\frac18\Big)
\sin\Big(2\pi f_j \frac{\sqrt q}{L}  \Big).
\]

Let $\epsilon_r=1$ for $r=1,\ldots, k$
and $\epsilon_r=-1$ for $r=k+1, \ldots, m-1$.
By the last condition on the sum we have $f_m=\sum_{r} \epsilon_r f_r$. Since we also have $f_m, f_1 \ge 1$
it follows
that $f_1 \ge 1+\max(0,- \sum_{2 \le j \le m-1} \epsilon_j f_j)$.
Let $\mathbf f=(f_2, \ldots, f_{m-1}) \in \mathbb Z^{m-2}$, and write $g(\mathbf f)=\max(0,-\sum_{2 \le j \le m-1}\epsilon_j f_j)$. Also, apply the bound
$d(n) \ll n^{\delta}$.
Thus,
\[
\mathbb E \Big(Y(q)^k \overline{Y(q)}^{\ell} \Big) \ll \frac{1}{q^{3m/4-m\delta}}\sum_{f_2, \ldots,  f_{m-1}\geq 1}
\sum_{\substack{f_1\geq 1+g(\mathbf f)}}
\frac{\Big|\sin\Big(2\pi (\sum_{ r } \epsilon_r f_r) \frac{\sqrt q}{L}  \Big)\Big|}{(\sum_{r} \epsilon_r f_r)^{3/2-2\delta}}  \prod_{j=1}^{m-1} \frac{\Big|\sin\Big(2\pi f_j \frac{\sqrt q}{L}  \Big)\Big|}{f_j^{3/2-2\delta}}.
\]
Since $q \le L^2$ the right-hand side is
\[
\ll \frac{1}{q^{3m/4-m\delta}}\int_{1}^{\infty} \cdots \int_1^{\infty}\int_{1+g(\mathbf x)}^{\infty}
\frac{\Big|\sin\Big(2\pi (\sum_{ r } \epsilon_r x_r) \frac{\sqrt q}{L}  \Big)\Big|}{(\sum_{r} \epsilon_r x_r)^{3/2-2\delta}}  \prod_{j=1}^{m-1} \frac{\Big|\sin\Big(2\pi x_j \frac{\sqrt q}{L}  \Big)\Big|}{x_j^{3/2-2\delta}} dx_1 \, dx_2 \ldots \, dx_{m-1}.
\]

Next, make the change of variables $u_j=x_j \sqrt{q}/L$ to see that
\begin{equation} \notag
\begin{split}
& \mathbb E \Big(Y(q)^k \overline{Y(q)}^{\ell} \Big) \ll \frac{1}{q^{\frac{(m-1)}{2}} L^{\frac{m}{2}+1-2m\delta}}
\int_{\frac{\sqrt{q}}{L}}^{\infty} \cdots \int_{\frac{\sqrt{q}}{L}}^{\infty} \int_{\frac{\sqrt{q}}{L}+g( \mathbf u)}^{\infty}
\prod_{j=1}^{m-1} \frac{\big|\sin\big(2\pi u_j  \big)\big|}{u_j^{3/2-2\delta}}  \\
& \qquad \qquad \qquad \qquad \qquad \qquad \qquad \times
\frac{\Big|\sin\Big(2\pi (\sum_{r} \epsilon_r u_r)  \Big)\Big|}{(\sum_{r} \epsilon_r u_r)^{3/2-2\delta}}
\, du_1 \, du_2 \ldots \,du_{m-1}.
\end{split}
\end{equation}
The multiple integral is $O(1)$. Hence, we have
for $q \le L^2$ that
\[
\mathbb E \Big(Y(q)^k \overline{Y(q)}^{\ell} \Big) \ll \frac{1}{q^{\frac{(m-1)}{2}} L^{\frac{m}{2}+1-2m\delta}}.
\]

\end{proof}

\begin{proof}[Proof of Proposition \ref{moments prop}]
By Lemma \ref{mean values} it suffices to estimate
\[
\mathbb E \Big( \Big(\frac{1}{\sigma_M} \cdot \sum_{q \le M} \tmop{Im} Y(q)\Big)^m\Big)= \frac{1}{\sigma_M^m} \cdot
\sum_{q_1, \ldots, q_m \le M} \mathbb E \Big(\prod_{j=1}^m \tmop{Im} Y(q_j)\Big).
\]
We analyze this sum in the following way. Consider a division of $\{1,\ldots,m\}$
into nonempty disjoint subsets $S_1,\ldots,S_n$ with cardinalities $\alpha_1,\ldots,\alpha_n$
such that $\alpha_1+\cdots+\alpha_n=m$. Given such a division, look at the contribution of the terms
in the above sum over $q_1,\ldots,q_m$ such that $q_a=q_b$ if $a,b \in S_j$ for some $j$
and $q_a \ne q_b$ if $a \in S_j$ and $b \in S_i$ with $i \ne j$.
Since the random variables $Y(q)$ are independent the sum of these terms equals
\[
\sum_{\substack{r_1, \dots, r_n \le M  \\ r_j \, \, \mbox{\tiny{distinct}}}}
\prod_{j=1}^n \mathbb E \Big(\Big(\frac{1}{\sigma_M} \tmop{Im} Y(r_j)\Big)^{\alpha_j}\Big).
\]
Thus,
\begin{equation} \label{comb}
\frac{1}{\sigma_M^m} \cdot
\sum_{q_1, \ldots, q_m \le M} \mathbb E \Big(\prod_{j=1}^m \tmop{Im} Y(q_j)\Big)
=\sum_{n=1}^{m} \sum_{\alpha_j} \frac{m!}{\alpha_1! \cdots \alpha_n!} \cdot \frac{1}{n!}
\sum_{\substack{r_1, \dots, r_n \le M  \\ r_j \, \, \mbox{\tiny{distinct}}}}
\prod_{j=1}^n \mathbb E \Big(\Big(\frac{1}{\sigma_M} \tmop{Im} Y(r_j)\Big)^{\alpha_j}\Big),
\end{equation}
where $\sum_{\alpha_j}$ runs over all $n$-tuples of positive integers $(\alpha_1,\ldots,\alpha_n)$
such that $\alpha_1 +\cdots+\alpha_n=m$. 

 Next note that inside the inner sum in \eqref{comb}  if $\alpha_j=1$ for some $j$ then this term vanishes.
Additionally, if $\ell \geq 3$ we have for any $\varepsilon>0$ that
\[
\mathbb E \Big(\sum_{q \le M}\Big(\frac{1}{\sigma_M} \tmop{Im} Y(q)\Big)^{\ell}\Big) \ll \frac{1}{L^{1-\varepsilon}}
\]
by Lemma \ref{Y bd}. Thus, by Lemma \ref{Y bd} each term in the inner sum on the right-hand side of
\eqref{comb} with  $\alpha_j \geq 3$ for some $j$
is $\ll 1/L^{1-\varepsilon}$. 
The contribution of all such terms is also $\ll 1/L^{1-\varepsilon}$.

The remaining terms have $\alpha_1=\ldots=\alpha_n=2$.
Since $\alpha_1+\cdots+\alpha_n=m$ we have that $m$ is even and $n=m/2$. Thus, the sum of these terms equals
\[
\frac{m!}{(2!)^{n}} \cdot \frac{1}{n!}
\sum_{\substack{r_1, \dots, r_n \le M  \\ r_j \, \, \mbox{\tiny{distinct}}}} \prod_{j=1}^n 
\mathbb E \Big(\Big(\frac{1}{\sigma_M} \tmop{Im} Y(r_j)\Big)^{2}\Big)
= \frac{m!}{2^{m/2} (m/2)!} \sum_{\substack{r_1, \dots, r_n \le M  \\ r_j \, \, \mbox{\tiny{distinct}}}} \prod_{j=1}^n 
\mathbb E \Big(\Big(\frac{1}{\sigma_M} \tmop{Im} Y(r_j)\Big)^{2}\Big).
\] 

To complete the proof we estimate the sum on the right-hand side.
Note by Lemma \ref{Y bd} that 
\[
0 \le \mathbb E \Big(\Big(\frac{1}{\sigma_M} \tmop{Im} Y(q)\Big)^{2}\Big) \ll
\begin{cases}
 \displaystyle \frac{1}{L^{2-\varepsilon}} \mbox{ if } q> L^2 \\
 \displaystyle \frac{1}{q^{1/2}L^{1-\varepsilon}} \mbox{ if } q \le L^2.
\end{cases}
\]
Thus,
\begin{equation} \notag
\begin{split}
&\bigg|\bigg(\mathbb E \Big(\sum_{q \le M }\Big(\frac{1}{\sigma_M} \tmop{Im} Y(q) \Big)^{2}\Big)\bigg)^n
-\sum_{\substack{r_1, \dots, r_n \le M  \\ r_n \neq r_{n-1}, \ldots, r_1}}
\prod_{j=1}^n \mathbb E \Big(\Big(\frac{1}{\sigma_M} \tmop{Im} Y(r_j)\Big)^{2}\Big)\bigg| \\
& \qquad \qquad \qquad \qquad \qquad \qquad \qquad \qquad \qquad \qquad \qquad 
\ll \frac{1}{L^{1-\varepsilon}} \, \, \mathbb E \Big(\sum_{q \le M }\Big(\frac{1}{\sigma_M} \tmop{Im} Y(q)\Big)^{2}\Big)^{n-1}.
\end{split}
\end{equation}
Iterating this argument, we see that
\begin{equation} \notag
\begin{split}
&\bigg|\bigg(\mathbb E \Big(\sum_{q \le M }\Big(\frac{1}{\sigma_M} \tmop{Im} Y(q) \Big)^{2}\Big)\bigg)^n
- \sum_{\substack{r_1, \dots, r_n \le M  \\ r_j \, \, \mbox{\tiny{distinct}}}}
\prod_{j=1}^n \mathbb E \Big(\Big(\frac{1}{\sigma_M} \tmop{Im} Y(r_j)\Big)^{2}\Big)\bigg| \\
& \qquad \qquad \qquad \qquad \qquad \qquad \qquad \qquad \qquad \qquad \qquad 
\ll \frac{1}{L^{1-\varepsilon}} \, \, \mathbb E \Big(\sum_{q \le M }\Big(\frac{1}{\sigma_M} \tmop{Im} Y(q)\Big)^{2}\Big)^{n-1}.
\end{split}
\end{equation}
Recalling the definition of $\sigma_M$, this gives
\[
\sum_{\substack{r_1, \dots, r_n \le M  \\ r_j \, \, \mbox{\tiny{distinct}}}}
\prod_{j=1}^n \mathbb E \Big(\frac{1}{\sigma_M} \tmop{Im} Y(r_j)\Big)^{2}=1+O(1/L^{1-\varepsilon}).
\]
Collecting estimates, we have
\begin{equation} \notag
\begin{split}
\mathbb E \Big( \Big(\frac{1}{\sigma_M} \cdot \sum_{q \le M} \tmop{Im} Y(q)\Big)^m\Big)
=
\begin{cases}
 \frac{m!}{2^{m/2} (m/2)! }+O\Big(\frac{1}{L^{1-\varepsilon}}\Big) \mbox{ if $m$ is even}, \\
O\Big(\frac{1}{L^{1-\varepsilon}}\Big) \mbox{ if $m$ is odd}.
\end{cases}
\end{split}
\end{equation}

\end{proof}

\subsection{The Proof of Theorem \ref{Main theorem}}

Let
\[
\mathbb P_{\omega, T}\Big( f(x) \in [\alpha, \beta]  \Big)=\int_{\mathbb R} \mathbf 1_{[\alpha, \beta]}(f(x)) \, \omega\Big(\frac{x}{T} \Big) \, \frac{dx}{T},
\]
where $\omega \geq0$ is a Schwartz function supported 
on the positive real numbers with unit mass.
By Proposition \ref{moments prop} it follows that for $M$ such that $M \ll T^{\varepsilon}$ for all $\varepsilon>0$
as $T \rightarrow \infty$
\[
\mathbb P_{\omega, T}\Big( \frac{1}{\sigma_M} S(x, L, M) \in [\alpha, \beta]  \Big)
=\frac{1}{\sqrt{2\pi}} \int_{\alpha}^{\beta} \! e^{-x^2/2} \, dx+o(1).
\]
Next note that by \eqref{sigma} we have $\sigma_M^2=\sigma^2+o(\sigma^2)$ if $L/\sqrt{M} \rightarrow 0$.
Thus, it follows that
\[
\mathbb P_{\omega, T}\Big( \frac{1}{\sigma} S(t, L, M) \in [\alpha, \beta]  \Big)
=\frac{1}{\sqrt{2\pi}} \int_{\alpha}^{\beta} \! e^{-t^2/2} \, dt+o(1).
\]
By Lemma \ref{conv prob} we have for any fixed $\epsilon>0$ that
\[
\mathbb P_{\omega, T}\Big( \frac{1}{\sigma} \Big(S(x, L)-S(x, L, M)\Big) > \epsilon  \Big)
\leq \frac{1}{\epsilon^2 \sigma^2}\bigg< \Big(S(x, L)-S(x, L, M)\Big)^2\bigg>=o(1)
\]
as $T \rightarrow \infty$. Thus, since $\epsilon>0$ is arbitrary we have as $T \rightarrow \infty$
\[
\mathbb P_{\omega, T}\Big( \frac{1}{\sigma} S(x, L) \in [\alpha, \beta]  \Big)
=\frac{1}{\sqrt{2\pi}} \int_{\alpha}^{\beta} \! e^{-t^2/2} \, dt+o(1).
\]

Let $\varepsilon>0$. We now choose $\omega$ such that $\omega(x)=1$ for $x \in[1+\varepsilon, 2-\varepsilon]$.
Since, $\omega\ge0$ and has unit mass we conclude that
\[
\int_{\mathbb R} \mathbf 1_{[\alpha, \beta]}\Big(\frac{1}{\sigma} S(x, L, M) \Big) \,
\Big(\mathbf 1_{[1,2]}\Big(\frac{x}{T} \Big)-\omega\Big(\frac{x}{T} \Big)\Big) \, \frac{dx}{T}
\leq \bigg(\int\limits_{-\infty}^{1+\varepsilon}+\int\limits_{2-\varepsilon}^{\infty} \bigg) \Big(\mathbf 1_{[1,2]}(x)+\omega(x)\Big) \, dx
\ll \varepsilon.
\]
Hence, as $T \rightarrow \infty$
\[
\frac{1}{T}\tmop{meas}\Big\{ x \in[T, 2T]: \frac{1}{\sigma} S(x, L) \in [\alpha, \beta]  \Big\}
=\frac{1}{\sqrt{2\pi}} \int_{\alpha}^{\beta} \! e^{-x^2/2} \, dx+o(1).
\]
To complete the proof first note that
for $2<h<x$ we have $\Delta(x+h)/(x+h)^{1/4}=\Delta(x+h)/x^{1/4}+O(h/x^{3/4})$,
since $\Delta(x)=O(x^{1/2})$. Next, observe that for $0< h \ll 1$ we have
$\Delta(x+h)=\Delta(x)+O(h \log x)+O(x^{\delta})$ for any $\delta>0$, by the bound
$d(n) \ll n^{\delta}$.  Thus,
as $T \rightarrow \infty$
\[
\frac{1}{\sigma} S(x, 2L)
=\frac{\Delta\Big(x+\frac{\sqrt{x}}{L}\Big)-\Delta(x)}{x^{1/4} \sqrt{\frac{8}{\pi^2} \frac{\log^3 L}{L}}}\left(1+o(1)\right)+o\left(1\right).
\]

\subsection{The proof of Theorem \ref{Main theorem Hecke}}

Let $f$ be a primitive cusp form of even weight
$k$ and level 1 with the previous notation. By Deligne's
bound we have $\left|\rho_f\left(n\right)\right|\leq d\left(n\right);$ the corresponding Dirichlet series
\[
\varphi\left(s\right)=\sum_{n=1}^{\infty}\frac{\rho_f\left(n\right)}{n^{s}}
\]
is absolutely convergent in the strip $\tmop{Re}\left(s\right)>1$,
and the classical theory shows that it has an analytic continuation
to the whole complex plane, with the functional equation
\[
\varphi\left(s\right)=\chi\left(s\right)\varphi\left(1-s\right),
\]
where
\[
\chi\left(s\right)=\left(-1\right)^{k/2}\left(2\pi\right)^{2s-1}\frac{\Gamma\left(\frac{k+1}{2}-s\right)}{\Gamma\left(s+\frac{k-1}{2}\right)}.
\]
We will use the following analog of formula (12.4.4) in \cite{Titchmarsh}, which is a special case of Theorem 1.2 in \cite{FriedIwaniec}: 

\begin{theorem}[Theorem 1.2 of \cite{FriedIwaniec} for $a\left(n\right)=\rho_f\left(n\right)$]
\label{thm:CuspVoronoi}For $x\geq1$, $N\le x$ and $\varepsilon>0$,
\begin{alignat*}{1}
\sum_{n\leq x}\rho_f\left(n\right) & =\frac{x^{1/4}}{\pi\sqrt{2}}\sum_{n\leq N}\rho_f\left(n\right)n^{-3/4}\cos\left(4\pi\sqrt{nx}-\frac{\pi}{4}\right)\\
 &\quad  +O_{\varepsilon,k}\left(\frac{x^{1/2+\varepsilon}}{\sqrt{N}}\right).
\end{alignat*}
\end{theorem}

We can now deduce the analogous results regarding the distribution
of the sum of the Hecke eigenvalues $\rho_f\left(n\right)$ in short intervals:
recall that
\[
A_{f}\left(x\right)=\sum_{n\leq x}\rho_f\left(n\right),
\]
and define $F_{f}\left(x\right)=x^{-1/4}A_{f}\left(x\right)$,    $S_{f}\left(x,L\right)=F_{f}\left(\left(\sqrt{x}+\frac{1}{L}\right)^2\right)-F_{f}\left(x\right)$, with the condition that $L\ll T^\varepsilon$ for all $\varepsilon>0$.
Note that the formula in Theorem \ref{thm:CuspVoronoi} and the classical
formula (12.4.4) in \cite{Titchmarsh} are the same except for
the coefficients $\rho_f\left(n\right)$ which replace $d\left(n\right)$.
The calculation for the expectation
of $S_{f}\left(x,L\right)$ goes line by line as in the divisor's
case (i.e. $\left\langle S_{f}\left(x,L\right)\right\rangle \ll T^{-1/4+\varepsilon})$,
and so is the calculation for the variance, until the part which uses formula \eqref{eq:varformula} which is now replaced by Rankin's result \eqref{Rankin Selberg}
We conclude that
\begin{alignat*}{1}
\left\langle S_{f}\left(x,L\right)^{2}\right\rangle  & \sim\frac{1}{\pi^{2}}\sum_{n\leq T^{1-\delta}}\frac{\rho_f^{2}\left(n\right)}{n^{3/2}}\sin^{2}\left(\frac{2\pi\sqrt{n}}{L}\right)\\
 & \sim\frac{c_{f}}{\pi^{2}}\int_{1}^{T^{1-\delta}}\frac{1}{x^{3/2}}\sin^{2}\left(\frac{2\pi\sqrt{x}}{L}\right)\mbox{d}x\\
 & =\frac{2c_{f}}{\pi^{2}L}\int_{1/L}^{T^{1/2-\delta/2}/L}\frac{1}{y^{2}}\sin^{2}\left(2\pi y\right)\mbox{d}y\\
 & \sim\frac{2c_{f}}{\pi^{2}L}\int_{0}^{\infty}\frac{1}{y^{2}}\sin^{2}\left(2\pi y\right)\mbox{d}y=\frac{2c_{f}}{L}.
\end{alignat*}
Denote by
\[
\sigma_{f}^{2}=\frac{2c_{f}}{L}
\]
the variance we calculated.

Again, for any $M=O\left(T^{1-\delta}\right)$ such that $L/\sqrt{M}\to0,$
we define the ``short'' approximation to $S_f(x,L)$ by
\[
S_{f}\left(x,L,M\right)=\frac{-2}{\pi\sqrt{2}}\sum_{n\leq M}\frac{\rho_f\left(n\right)}{n^{3/4}}\sin\left(\frac{2\pi\sqrt{n}}{L}\right)\sin\left(4\pi\sqrt{n}\left(\sqrt x+\frac{1}{2L}\right)-\frac{\pi}{4}\right).
\]
 By similar considerations $\left\langle S_{f}\left(x,L,M\right)^{2}\right\rangle \sim\sigma_{f}^{2}$; the proof that $\left\langle \left(\frac{S_{f}\left(x,L\right)-S_{f}\left(x,L,M\right)}{\sigma_{f}}\right)^{2}\right\rangle \to0$
as $T\to\infty$ is again similar, and so is the proof of the
rest of the analogous claims, including the calculation of the higher
moments, which does not use any special property of the divisor function except for the property $d\left(n\right)\ll n^{\varepsilon}$,
which is also true for $\rho_f\left(n\right)$ by Deligne's bound. We
conclude that $\frac{S_{f}\left(x,L\right)}{\sigma_{f}}$ has a standard Gaussian limiting distribution as $T\to\infty$, from which Theorem \ref{Main theorem Hecke}
follows in the same manner as before (the bound $A_f(x) \ll x^{1/2}$ easily follows from Theorem \ref{thm:CuspVoronoi}).

\section*{Acknowledgments}
We would like to thank Ze\'ev Rudnick for suggesting this problem as well as for numerous helpful discussions
and remarks. 

\section*{Funding}
The research leading to these results has received funding
from the European Research Council under the European Union's
Seventh Framework Programme (FP7/2007-2013) / ERC grant agreement
n$^{\text{o}}$ 320755.

\end{document}